\documentclass[12pt,a4paper,reqno]{amsart} 
\usepackage{footnote}
\usepackage{amssymb}
\usepackage{latexsym}
\usepackage{amsmath}
\usepackage{mathrsfs}
\usepackage[X2,T1]{fontenc}
\usepackage[applemac]{inputenc}
\usepackage{cite}
\usepackage{calc}                   

\newcommand{\scal}[2]{\langle #1,#2\rangle}

\newcommand{\cc}[1]{\mathbf C^{#1}}
\newcommand{\nn}[1]{\mathbf N^{#1}}
\newcommand{\rr}[1]{\mathbf R^{#1}}
\newcommand{\zz}[1]{\mathbf Z^{#1}}

\newcommand{\mabfs}{{\boldsymbol s}}
\newcommand{\mabfsi}{{\boldsymbol \sigma}}

\newcommand{\nm}[2]{\Vert #1\Vert _{#2}}

\newcommand{\sets}[2]{\{ \, #1\, ;\, #2\, \} }

\newcommand{\ep}{\varepsilon}
\newcommand{\fy}{\varphi}
\newcommand{\cdo}{\, \cdot \, }

\newcommand{\eabs}[1]{\langle #1\rangle}     

\newcommand{\vrum}{\vspace{0.1cm}}

\newcommand{\maclA}{\mathcal A}

\newcommand{\maclH}{\mathcal H}

\newcommand{\maclS}{\mathcal S}

\newcommand{\mascF}{\mathscr F}

\newcommand{\mascS}{\mathscr S}

\numberwithin{equation}{section}          

\newtheorem{thm}{Theorem}
\numberwithin{thm}{section}
\newtheorem*{tom}{\rubrik}
\newcommand{\rubrik}{}
\newtheorem{prop}[thm]{Proposition}

\newtheorem{lemma}[thm]{Lemma}

\theoremstyle{definition}

\newtheorem{defn}[thm]{Definition}
\newtheorem{example}[thm]{Example}

\theoremstyle{remark}
\newtheorem{rem}[thm]{Remark}

\author{Joachim Toft}

\address{Department of Mathematics,
Linn{\ae}us University, V{\"a}xj{\"o}, Sweden}

\email{joachim.toft@lnu.se}

\title{Tensor products for Gelfand-Shilov and Pilipovi{\'c} distribution spaces}



\begin{document}


\begin{abstract}
We show basic properties on tensor products for Gelfand-Shilov distributions
and Pilipovi{\'c} distributions. This also includes the Fubbini's property of such
tensor products. We also apply the Fubbini property to deduce some properties
for short-time Fourier transforms of Gelfand-Shilov and Pilipovi{\'c}
distributions.
\end{abstract}

\maketitle

\par

\section{Introduction}\label{sec0}

\par

An important issue in mathematics concerns tensor products. When
considering the functions $f_j$ defined on $\Omega _j\subseteq \rr {d_j}$, $j=1,2$,
and with values in $\mathbf C$, their tensor product
$f_1\otimes f_2$ is the function from $\Omega _1\times \Omega _2$ to $\mathbf C$
given by the formula
$$
(f_1\otimes f_2) (x_1,x_2) =f_1(x_1)f_2(x_2),\qquad x_j\in \Omega _j,\ j=1,2.
$$
Let $f_j,\fy _j\in \mascS (\rr {d_j})$, $f=f_1\otimes f_2$,
$\fy \in \mascS (\rr {d_1+d_2})$, and let $\psi _1$ and $\psi _2$
be given by
\begin{equation}\label{Eq:psijFunctions}
\psi _1(x_1) = \scal {f_2}{\fy (x_1,\cdo )}
\quad \text{and}\quad
\psi _2(x_2) = \scal {f_1}{\fy (\cdo ,x_2)},
\end{equation}
(For notations, see \cite{Ho1} and Section \ref{sec1}.) Then it follows that
\begin{gather}
\scal f{\fy _1\otimes \fy _2} =  \scal {f_1}{\fy _1}\scal {f_2}{\fy _2},
\label{Eq:TensorTensor1}
\intertext{and that the Fubbini's property}
\scal f \fy = \scal {f_1}{\psi _1} = \scal {f_2}{\psi _2}
\label{Eq:TensorTensor2}
\end{gather}
holds.

\par

The formulae \eqref{Eq:TensorTensor1} and \eqref{Eq:TensorTensor2}
are essential when searching for extention of tensor products to distributions.
By the analysis in \cite[Chapter V and VII]{Ho1}, we have the following.

\par

\begin{thm}\label{Thm:SchwartzTensor}
Let $f_j\in \mascS '(\rr {d_j})$, $\fy \in \mascS (\rr {d_1+d_2})$
and let $\psi _j$ be given by \eqref{Eq:psijFunctions},
$j=1,2$. Then $\psi _j\in  \mascS (\rr {d_2})$, $j=1,2$, and
there is a unique $f\in \mascS '(\rr {d_1+d_2})$ such that
for every $\fy _1\in \mascS (\rr {d_1})$ and
$\fy _2\in \mascS (\rr {d_2})$, \eqref{Eq:TensorTensor1} and
\eqref{Eq:TensorTensor2} hold.
\end{thm} 

\par

The existence of a distribution $f$ in the previous theorem which satisfies
\eqref{Eq:TensorTensor1} can also be deduced by a general and
abstract result on tensor products for nuclear spaces
(see \cite[Chapter 50]{Tre}). On the
other hand, in order to reach the Fubbini property \eqref{Eq:TensorTensor2}, it seems
that more structures are needed.

\par

A more specific approach in the lines of the ideas in \cite{Tre}
is indicated in \cite{LozPer,ReSi}, where
$\mascS (\rr d)$ and $\mascS '(\rr d)$ are described by suitable series
expansions of Hermite functions. By following such approach, the situations
are essentially reduced to
questions on tensor products of weighted $\ell ^2$ spaces, and both
properties \eqref{Eq:TensorTensor1} and \eqref{Eq:TensorTensor2} follows
from such approach.

\par

In Sections \ref{sec2} and \ref{sec3} we show that Theorem \ref{Thm:SchwartzTensor}
holds in the context of Gelfand-Shilov spaces, Pilipovi{\'c} spaces and
their distribution (dual) spaces. In particular, we prove that the following results hold true.

\par

\begin{thm}\label{Prop:GSTensor}
Let $s_j,\sigma _j>0$, $f_j\in (\maclS _{s_j}^{\sigma _j})'(\rr {d_j})$,
$\fy \in \maclS _{s_1,s_2}^{\sigma _1,\sigma _2}(\rr {d_1+d_2})$
and let $\psi _j$ be given by \eqref{Eq:psijFunctions},
$j=1,2$.
Then $\psi _j\in  \maclS _{s_j}^{\sigma _j}(\rr {d_2})$, $j=1,2$, and
there is a unique $f\in
(\maclS _{s_1,s_2}^{\sigma _1,\sigma _2})'(\rr {d_1+d_2})$ such that
for every $\fy _1\in \maclS _{s_1}^{\sigma _1}(\rr {d_1})$ and
$\fy _2\in \maclS _{s_2}^{\sigma _2}(\rr {d_2})$,
\eqref{Eq:TensorTensor1} and
\eqref{Eq:TensorTensor2} hold.

\par

The same holds true with 
$\Sigma _{s_j}^{\sigma _j}$, $(\Sigma _{s_j}^{\sigma _j})'$,
$\Sigma _{s_1,s_2}^{\sigma _1,\sigma _2}$ and $(\Sigma _{s_1,s_2}^{\sigma _1,\sigma _2})'$
in place of
$\maclS _{t_j}^{s_j}$, $(\maclS _{s_j}^{\sigma _j})'$,
$\maclS _{s_1,s_2}^{\sigma _1,\sigma _2}$ and $(\maclS _{s_1,s_2}^{\sigma _1,\sigma _2})'$,
respectively, at each occurrence.
\end{thm}

\par

\begin{thm}\label{Thm:TensorprodDistr}
Let $s\in \overline{\mathbf R_\flat}$ and let $f_j\in \maclH _s'(\rr {d_j})$,
$\fy \in  \maclH _s(\rr {d_1+d_2})$
and let $\psi _j$ be given by \eqref{Eq:psijFunctions},
$j=1,2$. Then $\psi _j\in   \maclH _s(\rr {d_2})$, $j=1,2$, and
there is a unique $f\in \maclH _s'(\rr {d_1+d_2})$ such that
for every $\fy _1\in \maclH _s(\rr {d_1})$ and
$\fy _2\in \maclH _s(\rr {d_2})$, \eqref{Eq:TensorTensor1} and
\eqref{Eq:TensorTensor2} hold.

\par

The same holds true with $\maclH _{0,s}$ and $\maclH _{0,s}'$ in place of
$\maclH _s$ and $\maclH _s'$, respectively, at each occurrence.
\end{thm}

\par

The distribution $f$ in Theorem \ref{Thm:SchwartzTensor}, Theorem
\ref{Prop:GSTensor} or in Theorem \ref{Thm:TensorprodDistr} is called
the \emph{tensor product} of $f_1$ and $f_2$ and is denoted by $f_1\otimes f_2$
as before.

\par

We remark that Gelfand-Shilov spaces of functions and distributions appear naturally when
discussing analyticity and well-posedness of solutions to partial differential equations (cf.
\cite {CaRo,CaTo}). Pilipovi{\'c}
spaces of functions and distributions often agree with Fourier-invariant Gelfand-Shilov spaces,
and possess convenient mapping properties with respect to the Bargmann transform. They therefore
seems to be suitable to have in background on problems in partial differential equations
which have been transformed by the Bargmann transform (see \cite{FeGaTo,Toft15} for
more details).

\par

Since the spaces in Theorems \ref{Prop:GSTensor} and \ref{Thm:TensorprodDistr}
are unions and intersections of nuclear spaces, the existence of $f$ satisfying
\eqref{Eq:TensorTensor1}
may be deduced by the abstract analogous results in \cite{Tre}.
Some parts of Theorem \ref{Prop:GSTensor} are also proved in
\cite{LozPer}.

\par

In Section \ref{sec2} we give a proof of Theorem \ref{Prop:GSTensor},
by using the framework in \cite{Ho1} for the proof of
Theorem \ref{Thm:SchwartzTensor}. In Section \ref{sec3} we use that
Pilipovi{\'c} spaces and their distribution spaces can be described by unions
and intersections of Hilbert spaces of Hermite series expansions. In similar
ways as in \cite{ReSi}, this essentially reduce the situation to deal with questions
on tensor products of weighted $\ell ^2$ spaces.

\par

In the end of Section \ref{sec2} we also give example on how to apply the Fubbini
property \eqref{Eq:TensorTensor2} to deduce certain relations for short-time Fourier
transforms (which often called coherent state trasnform in physics) of Gelfand-Shilov
distributions (see Example \ref{Ex:STFTforGS}). In Section \ref{sec3} we also discuss
such questions for Pilipovi{\'c} spaces which are not Gelfand-Shilov distributions
(cf. Remark \ref{Rem:STFTforPilSpaces}).

\par

\section{Preliminaries}\label{sec1}

\par

In this section we recall some basic facts. We start by giving
the definition of Gelfand-Shilov spaces. Thereafter we recall
some the definition of Pilipovi{\'c} spaces and recall some of their properties.

\par

\subsection{Gelfand-Shilov spaces} 

\par

We start by recalling some facts about Gelfand-Shilov spaces (cf. \cite{GS}).
Let $0<h,s_j,\sigma _j\in \mathbf R$, $j=1,\dots ,n$, be fixed,
$d=d_1+\cdots d_n$, where $d_j\ge 0$ are integers, and let
$$
\mabfs =(s_1,\dots ,s_n)\in \rr n_+
\quad \text{and}\quad
\mabfsi =(\sigma _1,\dots ,\sigma _n)\in \rr n_+.
$$
For multi-indices of multi-indices we let
\begin{alignat*}{2}
\alpha !^{\mabfs} &= \alpha _1!^{s_1}\cdots \alpha _n!^{s_n}, &
\quad
x^\alpha &=x^{\alpha _1}\cdots x^{\alpha _n},
\\[1ex]
D_x^\alpha &= D_{x_1}^{\alpha _1}\cdots D_{x_n}^{\alpha _n} &
\qquad \text{and}\qquad
|\alpha | &= |\alpha _1|+\cdots +|\alpha _n|
\end{alignat*}
when
$$
\alpha = (\alpha _1,\dots ,\alpha _n)\in \nn {d_1}\times \cdots \times \nn {d_n},
\quad \text{and}\quad
x=(x_1,\dots ,x_n)\in \rr {d_1}\times \cdots \times \rr {d_n}.
$$
For any $f\in C^\infty (\rr d)$, we let
\begin{equation}\label{gfseminorm}
\nm f{\maclS _{\mabfs ;h}^{\mabfsi}}
\\[1ex]
\equiv
\sup
\left (
\frac {\nm {x^{\alpha}
\partial _{x}^\beta f}{L^\infty (\rr d)}}
{
h^{|\alpha  |+|\beta |}
\alpha !^{\mabfs}\, \beta !^{\mabfsi}
}
\right ),
\end{equation}
where the supremum is taken over all $\alpha _j,\beta _j\in \nn {d_j}$, $j=1,\dots ,d$.
Then $f\mapsto \nm f{\maclS _{\mabfs ;h}^{\mabfsi}}$ defines a norm on $C^\infty (\rr d)$
which might attend infinity.
The space
$\maclS _{\mabfs ;h}^{\mabfsi}(\rr {d})$ is
the Banach space which consist of all $f\in C^\infty (\rr {d})$ such that
$\nm f{\maclS _{\mabfs ;h}^{\mabfsi}}$ is finite.
In the case $d_1=d\ge 1$,
$d_2=\cdots =d_n=0$, $s=s_1$, $\sigma =\sigma _1$ and $x_1=x$, \eqref{gfseminorm}
is interpreted as
\begin{equation}\tag*{( \ref{gfseminorm} )$'$}
\nm f{\maclS _{s;h}^{\sigma}}
\equiv
\sup _{\alpha ,\beta \in
\mathbf N^{d}} 
\left (
\frac {\nm {x^{\alpha}
\partial _x^{\beta} f(x)}{L^\infty (\rr d)}}
{h^{|\alpha  +\beta |}
\alpha !^{s}\, \beta !^{\sigma}}
\right ).
\end{equation}

\par

The \emph{Gelfand-Shilov spaces} $\maclS _{\mabfs}^{\mabfsi}
(\rr {d})$ and $\Sigma _{\mabfs}^{\mabfsi} (\rr {d})$
are defined as the inductive and projective 
limits respectively of $\maclS _{\mabfs ;h}^{\mabfsi}(\rr {d})$.
This implies that
\begin{equation}\label{GSspacecond1}
\begin{aligned}
\maclS _{\mabfs}^{\mabfsi} (\rr {d})
&=
\bigcup _{h>0}
\maclS _{\mabfs ;h}^{\mabfsi} (\rr {d}),
\\[1ex]
\Sigma _{\mabfs}^{\mabfsi} (\rr {d})
&=
\bigcap _{h>0}
\maclS _{\mabfs ;h}^{\mabfsi} (\rr {d}),
\end{aligned}
\end{equation}
and that the topology for $\maclS _{\mabfs}^{\mabfsi} (\rr {d})$
is the strongest
possible one such that the inclusion map from $\maclS _{\mabfs ;h}^{\mabfsi}
(\rr {d})$ to $\maclS _{\mabfs}^{\mabfsi} (\rr {d})$
is continuous, for every choice 
of $h>0$. The space $\Sigma _{\mabfs}^{\mabfsi} (\rr {d})$ is a Fr{\'e}chet space
with seminorms $\nm \cdo {\maclS _{\mabfs ;h}^{\mabfsi}}$,
$h>0$. Moreover,
\begin{alignat*}{3}
\Sigma _{\mabfs}^{\mabfsi} (\rr {d})&\neq \{ 0\} & \quad
&\Leftrightarrow & \quad
s_j+\sigma _j &\ge 1\ \text{and}\ (s_j,\sigma _j)\neq (\frac 12,\frac 12),\ j=1,\dots ,n,
\intertext{and}
\maclS _{\mabfs}^{\mabfsi} (\rr {d})&\neq \{ 0\} & \quad
&\Leftrightarrow & \quad
s_j+\sigma _j &\ge 1,\ j=1,\dots ,n.
\end{alignat*}

\par

There are various kinds of characterisations of the spaces
$\maclS _{\mabfs}^{\mabfsi} (\rr {d})$ and
$\Sigma _{\mabfs}^{\mabfsi} (\rr {d})$, e.{\,}g.
in terms of the exponential decay of their elements. Later on it will
be useful that
$f \in \maclS _{\mabfs}^{\mabfsi} (\rr {d})$
(respectively $f \in \Sigma _{\mabfs}^{\mabfsi} (\rr {d})$), if and only if 
$$
|\partial _x^{\alpha}f(x)|
\lesssim
h^{|\alpha |} \alpha !^{\mabfsi} 
e^{-r(|x_1|^{\frac 1{s_1}} + \cdots +|x_n|^{\frac 1{s_n}})}
$$
for some $h,r>0$ (respectively for every $h>0, \ep>0$). 

\par

If $\boldsymbol 1=(1,\dots ,1)\in \rr n$ and $\mabfs ,\mabfsi \in \rr n_+$, then
\begin{equation}\label{Eq:GSEmbeddings}
\Sigma _{\mabfs}^{\mabfsi} (\rr {d})
\hookrightarrow
\maclS _{\mabfs}^{\mabfsi} (\rr {d})
\hookrightarrow
\Sigma _{\mabfs +\ep \boldsymbol 1}^{\mabfsi +\ep \boldsymbol 1} (\rr {d})
\hookrightarrow
\mascS (\rr {d})
\end{equation}
for every $\ep >0$. If in addition $s_j+\sigma _j\ge 1$ for every $j$, then
the last two inclusions in \eqref{Eq:GSEmbeddings} are dense,
and if in addition $(s_j,\sigma _j)\neq (\frac 12,\frac 12)$ for every $j$, then
the first inclusion in \eqref{Eq:GSEmbeddings} is dense.

\medspace

The \emph{Gelfand-Shilov distribution spaces} $(\maclS
_{\mabfs}^{\mabfsi} (\rr {d})$
and $(\Sigma _{\mabfs}^{\mabfsi} (\rr {d})$ are the
projective and inductive limit respectively of $(\maclS
_{\mabfs}^{\mabfsi})'(\rr {d})$.  This means that
\begin{equation}\tag*{(\ref{GSspacecond1})$'$}
\begin{aligned}
(\maclS _{\mabfs}^{\mabfsi} )'(\rr {d})
= \bigcap _{h>0}(\maclS _{\mabfs ;h}^{\mabfsi})'(\rr {d}),
\\[1ex]
(\Sigma _{\mabfs}^{\mabfsi} )'(\rr {d})
=\bigcup _{h>0}(\maclS _{\mabfs ;h}^{\mabfsi} )'(\rr {d}).
\end{aligned}
\end{equation}
If in addition $d_1=d\ge 1$, $d_2=\cdots =d_n=0$, $s=s_1$ and $\sigma _1=\sigma$,
then we set $(\maclS _s^\sigma )'(\rr d) =(\maclS _{\mabfs}^{\mabfsi})'
(\rr {d})$ and
$(\Sigma _s^\sigma )'(\rr d) =(\Sigma _{\mabfs}^{\mabfsi} )'
(\rr {d})$.
We remark that the analysis in \cite{Pil2} shows that $(\maclS
_{\mabfs}^{\mabfsi} )'(\rr {d})$
is the dual of $\maclS _{\mabfs}^{\mabfsi} (\rr {d})$, and
that $(\Sigma _{\mabfs}^{\mabfsi} )'(\rr {d})$
is the dual of $\Sigma _{\mabfs}^{\mabfsi} (\rr {d})$
(also in topological sense). By the inequalities $n!k! \le (n+k)! \le 2^{n+k}n!k!$
it follows that
\begin{alignat*}{2}
\maclS _{s,\dots ,s}^{\sigma ,\dots ,\sigma} (\rr {d})
&=
\maclS _{s}^{\sigma} (\rr {d}), &
\quad
\Sigma _{s,\dots ,s}^{\sigma ,\dots ,\sigma} (\rr {d})
&=
\Sigma _{s}^{\sigma} (\rr {d}),
\\[1ex]
(\maclS _{s,\dots ,s}^{\sigma ,\dots ,\sigma} )' (\rr {d})
&=
(\maclS _{s}^{\sigma})' (\rr {d_1+d_2}), &
\quad
(\Sigma _{s,\dots ,s}^{\sigma ,\dots ,\sigma} )' (\rr {d})
&=
(\Sigma _{s}^{\sigma})' (\rr {d_1+d_2}),
\end{alignat*}

\par

Corresponding relations in \eqref{Eq:GSEmbeddings} for 
Gelfand-Shilov distributions are
\begin{align*}
\mascS '(\rr {d}) &\hookrightarrow 
(\Sigma _{\mabfs +\ep \boldsymbol 1}^{\mabfsi +\ep \boldsymbol 1} )'(\rr {d})
\hookrightarrow
(\maclS _{\mabfs}^{\mabfsi} )'(\rr {d})
\intertext{when $s_j+\sigma _j \ge 1$, $j=1,\dots ,n$, and}
(\maclS _{\mabfs}^{\mabfsi} )'(\rr {d})
&\hookrightarrow (\Sigma _{\mabfs}^{\mabfsi} )'(\rr {d})
\end{align*}
when $s_j+\sigma _j \ge 1$ and $(s_j,\sigma _j)\neq (\frac 12,\frac 12)$,
$j=1,\dots ,n$.

\par

The Gelfand-Shilov spaces possess several convenient mapping
properties. For example they are invariant under
translations, dilations, and to some extent (partial) Fourier transformations. For any
$f\in L^1(\rr d)$, its Fourier transform is defined by
$$
(\mascF f)(\xi ) = \widehat f(\xi ) \equiv (2\pi )^{-\frac d2}\int _{\rr {d}} f(x)e^{-i\scal x\xi}\, dx.
$$
If instead $f\in L^1(\rr {d_1+\cdots +d_n})$, then the partial Fourier transform of $f$ with respect to
$k\in \{ 1,\dots ,n\}$ is given by
$$
(\mascF _{k}f)(x_1,\dots ,\xi _{k},\dots ,x_n) \equiv (2\pi )^{-\frac {d_{k}}2}\int _{\rr {d_{k}}}
f(x_1,\dots ,x_n)e^{-i\scal {x_{k}}{\xi _{k}}}\, dx_{k}, \quad x_j,\xi _j\in \rr {d_j}.
$$

\par

\begin{rem}\label{Rem:FTExtentions}
Let $d=d_1+\cdots +d_n$, $j\in \{ 1,\dots ,n\}$, $\mabfs ,\mabfsi \in \rr n_+$,
$\tau _{k,j}(\mabfs ,\mabfsi ) = (r_{k,j,1},\dots ,r_{k,j,n})$, $k=1,2$,
where
$$
r_{1,j,l}
=
\begin{cases}
s_l, & l\neq j
\\[1ex]
\sigma _l, & l=j,
\end{cases}
\quad \text{and} \quad
r_{2,j,l}
=
\begin{cases}
\sigma _l, & l\neq j
\\[1ex]
s _l, & l=j.
\end{cases}
$$
Then the following follows from the general theory of Schwartz functions and Gelfand-Shilov
functions and their distributions (see e.{\,}g. \cite{ChuChuKim,Ho1}):
\begin{enumerate}
\item the definition of $\mascF _j$ extends to a homeomorphism on
$\mascS '(\rr d)$ and restricts to a homeomorphism on $\mascS (\rr d)$;

\vrum

\item the definition of $\mascF _j$ extends uniquely to a homeomorphism
from $(\maclS _{\mabfs}^{\mabfsi})'(\rr d)$
to $(\maclS _{\tau _{1,j}(\mabfs ,\mabfsi )}^{\tau _{2,j}(\mabfs ,\mabfsi )})'(\rr d)$, and
from $(\Sigma _{\mabfs}^{\mabfsi})'(\rr d)$ to
$(\Sigma _{\tau _{1,j}(\mabfs ,\mabfsi )}^{\tau _{2,j}(\mabfs ,\mabfsi )})'(\rr d)$;

\vrum

\item $\mascF _j$ restricts to homeomorphisms from $\maclS _{\mabfs}^{\mabfsi}(\rr d)$
to $\maclS _{\tau _{1,j}(\mabfs ,\mabfsi )}^{\tau _{2,j}(\mabfs ,\mabfsi )}(\rr d)$, and from
$\Sigma _{\mabfs}^{\mabfsi}(\rr d)$
to $\Sigma _{\tau _{1,j}(\mabfs ,\mabfsi )}^{\tau _{2,j}(\mabfs ,\mabfsi )}(\rr d)$.
\end{enumerate}
\end{rem}

\par

\subsection{Pilipovi{\'c} spaces} 

\par

Next we make a review of Pilipovi{\'c} spaces. These spaces can be defined
in terms of Hermite series expansions. We recall that
the Hermite function of order $\alpha \in \nn d$ is defined by
$$
h_\alpha (x) = \pi ^{-\frac d4}(-1)^{|\alpha |}
(2^{|\alpha |}\alpha !)^{-\frac 12}e^{\frac {|x|^2}2}
(\partial ^\alpha e^{-|x|^2}).
$$
It follows that
$$
h_{\alpha}(x)=   ( (2\pi )^{\frac d2} \alpha ! )^{-1}
e^{-\frac {|x|^2}2}p_{\alpha}(x),
$$
for some polynomial $p_\alpha$ on $\rr d$, which is
called the Hermite polynomial of order $\alpha$. 
The Hermite functions are eigenfunctions to the Fourier transform, and
to the Harmonic oscillator $H_d\equiv |x|^2-\Delta$ which acts on functions
and (ultra-)distributions defined
on $\rr d$. More precisely, we have
$$
H_dh_\alpha = (2|\alpha |+d)h_\alpha ,\qquad H_d\equiv |x|^2-\Delta .
$$

\par

It is well-known that
the set of Hermite functions is a basis for $\mascS (\rr d)$ 
and an orthonormal basis for $L^2(\rr d)$ (cf. \cite{ReSi}).
In particular, if $f\in L^2(\rr d)$, then
$$
\nm f{L^2(\rr d)}^2 = \sum _{\alpha \in \nn d}|c_h(f,\alpha )|^2,
$$
where
\begin{align}
f(x) &= \sum _{\alpha \in \nn d}c_h(f,\alpha )h_\alpha ,
\label{Eq:HermiteExp}
\intertext{is the Hermite seriers expansion of $f$, and}
c_h(f,\alpha ) &= (f,h_\alpha )_{L^2(\rr d)}
\label{Eq:HermiteCoeff}
\end{align}
is the Hermite coefficient of $f$ of order $\alpha \in \rr d$.

\par

In order to define the full
scale of Pilipovi{\'c} spaces, their order $s$ should belong to
the extended set
$$
\mathbf R_\flat = \mathbf R_+\bigcup \sets {\flat _\sigma}{\sigma \in \mathbf R_+},
$$
of $\mathbf R_+$, with extended inequality relations as
$$
s_1<\flat _\sigma <s_2
\quad \text{and}\quad \flat _{\sigma _1}<\flat _{\sigma _2}
$$
when $s_1<\frac 12\le s_2$ and $\sigma _1<\sigma _2$. (Cf. \cite{Toft15}.)

\par

For $r>0$ and $s\in \mathbf R_\flat$ we set
\begin{align}
\vartheta _{r,s}(\alpha ) &\equiv
\begin{cases}
e^{-r|\alpha |^{\frac 1{2s}}}, &  s\in \mathbf R_+ ,
\\[1ex]
r^{|\alpha |}\alpha !^{-\frac 1{2\sigma }}, &  s = \flat _\sigma ,
\end{cases}
\label{varthetarsDef}
\intertext{and}
\vartheta _{r,s}'(\alpha ) &\equiv
\begin{cases}
e^{r|\alpha |^{\frac 1{2s}}}, &  s\in \mathbf R_+ ,
\\[1ex]
r^{|\alpha |}\alpha !^{\frac 1{2\sigma }}, &  s = \flat _\sigma .
\end{cases}
\label{varthetarsDualDef}
\end{align}

\par

\begin{defn}\label{Def:PilSpaces}
Let $s\in \overline{\mathbf R_\flat} = \mathbf R_\flat \cup \{ 0\}$,
and let $\vartheta _{r,s}$ and $\vartheta _{r,s}'$ be as in
\eqref{varthetarsDef} and \eqref{varthetarsDualDef}.
\begin{enumerate}
\item $\maclH _0(\rr d)$ consists of all Hermite polynomials, and
$\maclH _0'(\rr d)$ consists of all formal Hermite series expansions
in \eqref{Eq:HermiteExp};

\vrum

\item if $s\in \mathbf R_\flat$, then $\maclH _s(\rr d)$ ($\maclH _s^0(\rr d)$)
consists of all $f\in L^2(\rr d)$ such that
$$
|c_h(f,h_\alpha )| \lesssim \vartheta _{r,s}(\alpha )
$$
holds true for some $r\in \mathbf R_+$ (for every $r\in \mathbf R_+$);

\vrum

\item if $s\in \mathbf R_\flat$, then $\maclH _s'(\rr d)$ ($(\maclH _s^0)'(\rr d)$)
consists of all formal Hermite series expansions
in \eqref{Eq:HermiteExp} such that 
$$
|c_h(f,h_\alpha )| \lesssim \vartheta _{r,s}'(\alpha )
$$
holds true for every $r\in \mathbf R_+$ (for some $r\in \mathbf R_+$).
\end{enumerate}
The spaces $\maclH _s(\rr d)$ and $\maclH _s^0(\rr d)$ are called
\emph{Pilipovi{\'c} spaces of Roumieu respectively Beurling types} of order $s$,
and
$\maclH _s'(\rr d)$ and $(\maclH _s^0)'(\rr d)$ are called
\emph{Pilipovi{\'c} distribution spaces of Roumieu respectively Beurling types}
of order $s$.
\end{defn}

\par

\begin{rem}\label{Remark:GSHermite}
Let $\mascS _s(\rr d)$ and $\Sigma _s(\rr d)$
be the Fourier invariant Gelfand-Shilov spaces of order $s\in \mathbf R_+$
and of Rourmeu and Beurling types respectively (see \cite{Toft15} for notations).
Then it is proved in \cite{Pil1,Pil2} that
\begin{alignat*}{2}
\maclH _s^0(\rr d) &=\Sigma _s(\rr d)\neq \{ 0\} ,& \quad s&> \frac 12,
\\[1ex]
\maclH _s^0(\rr d) &\neq\Sigma _s(\rr d) = \{ 0\} ,& \ s&\le \frac 12,
\\[1ex]
\maclH _s(\rr d) &=\maclS _s(\rr d)\neq \{ 0\} ,& \quad s&\ge \frac 12
\intertext{and}
\maclH _s(\rr d) &\neq \maclS _s(\rr d)= \{ 0\} ,& \quad s&< \frac 12.
\end{alignat*}
\end{rem}

\par


\par

Next we recall the topologies for Pilipovi{\'c} spaces. Let $s\in \mathbf R_\flat$,
$r>0$, and let $\nm f{\maclH _{s;r}}$ and $\nm f{\maclH _{s;r}'}$ be given by
\begin{alignat}{2}
\nm f{\maclH _{s;r}}
&\equiv
\sup _{\alpha \in \nn d} |c_h(f,\alpha )/\vartheta _{r,s}(\alpha )| ,&
\quad 
s&\in \mathbf R_\flat ,
\label{Eq:PilSpaceSemiNorm}
\intertext{and}
\nm f{\maclH _{s;r}'}
&\equiv
\sup _{\alpha \in \nn d} |c_h(f,\alpha )/\vartheta _{r,s}'(\alpha )| ,&
\quad 
s&\in \mathbf R_\flat .
\label{Eq:PilDistSpaceSemiNorm}
\end{alignat}
when $f$ is a formal expansion in \eqref{Eq:HermiteExp}.
Then $\maclH _{s;r}(\rr d)$ consists of all expansions \eqref{Eq:HermiteExp}
such that $\nm f{\maclH _{s;r}}$ is finite, and $\maclH _{s;r}'(\rr d)$
consists of all expansions \eqref{Eq:HermiteExp}
such that $\nm f{\maclH _{s;r}'}$ is finite. It follows that both
$\maclH _{s;r}(\rr d)$ and $\maclH _{s;r}'(\rr d)$ are Banach spaces under
the norms $f\mapsto \nm f{\maclH _{s;r}}$ and $f\mapsto \nm f{\maclH _{s;r}'}$,
respectively.

\par

We let the topologies of $\maclH _s(\rr d)$ and $\maclH _s^0(\rr d)$
be the inductive respectively projective limit topology of $\maclH _{s;r}(\rr d)$
with respect to $r>0$. In the same way, the topologies of $\maclH _s'(\rr d)$
and $(\maclH _s^0)'(\rr d)$ are
the projective respectively inductive limit topology of $\maclH _{s;r}'(\rr d)$
with respect to $r>0$.

Suppose instead $s=0$. For any integer $N\ge 0$, we set
$$
\nm f{(0,N)} \equiv \sup _{|\alpha \le N|} |c_\alpha (f)|,\qquad f\in \maclH _0'(\rr d).
$$
The topology for $\maclH _0'(\rr d)$ is defined by the semi-norms
$\nm \cdo{(0,N)}$.

\par

We also let $\maclH _{0,N}(\rr d)$ be the vector space which consists of all
$f\in \maclH _0'(\rr d)$ such that $c_\alpha (f)=0$ when $|\alpha |>N$, and
equip this space with the topology, defined by the norm $\nm \cdo{(0,N)}$.
The topology of $\maclH _0(\rr d)$ is then defined as the inductive limit topology
of $\maclH _{0,N}(\rr d)$ with respect to $N\ge 0$.

\par

It follows that all the spaces in Definition
\ref{Def:PilSpaces} are complete, and that $\maclH _s^0(\rr d)$
and $\maclH _s'(\rr d)$ are Fr{\'e}chet space with semi-norms
$f\mapsto \nm f{\maclH _{s;r}}$ and $f\mapsto \nm f{\maclH _{s;r}'}$,
respectively.

\par

The following characterisations of Pilipovi{\'c} spaces can be found in \cite{Toft15}.
The proof is therefore omitted.

\par

\begin{prop}\label{Prop:PilSpaceChar}
Let $s\in \mathbf R_+\cup \{ 0\}$ and let $f\in \maclH _0'(\rr d)$.
Then $f\in \maclH _s^0(\rr d)$ ($f\in \maclH _s(\rr d)$), if and only if
$f\in C^\infty (\rr d)$ and
satisfies $|H_d^N f(x)|\lesssim h^N N!^{2s}$ for every $h>0$ (for some $h>0$).
\end{prop}

\par

Finally we remark that the Pilipovi{\'c} spaces of functions and distributions
possess convenient mapping properties under the Bargmann transform
(cf. \cite{Toft15}).

\par


%

\par

\section{Tensor product for Gelfand-Shilov
spaces}\label{sec2}

\par

In this section we start by proving Theorem \ref{Prop:GSTensor}.
Thereafter we deduce a multi-linear version of this result.

\par

For the proof of Theorem \ref{Prop:GSTensor} we first need the following
analogy of Lemma 4.1.3 in \cite{Ho1}.

\par

\begin{lemma}\label{Lemma:A.1}
Let $s_1,s_2,\sigma _1,\sigma _2>0$, $\fy ,\psi \in \maclS
_{s_1,s_2}^{\sigma _1,\sigma _2}(\rr {d_1+d_2})$.
Then the Riemann sum
$$
\sum _{k\in \zz d}\fy (x-\ep k)\psi (\ep k)\ep ^d,\quad d=d_1+d_2,
$$
converges to $(\fy *\psi )(x)$ in $\maclS
_{s_1,s_2}^{\sigma _1,\sigma _2}(\rr {d_1+d_2})$ as $\ep \to 0$.

\par

The same holds true if each
$\maclS _{s_1,s_2}^{\sigma _1,\sigma _2}$ and $(\maclS _{s_1,s_2}
^{\sigma _1,\sigma _2})'$ are replaced by
$\Sigma _{s_1,s_2}^{\sigma _1,\sigma _2}$ and $(\Sigma _{s_1,s_2}
^{\sigma _1,\sigma _2})'$,
respectively.
\end{lemma}

\par

\begin{proof}
We may assume that $\ep >0$, and consider first the case when
$\fy$ and $\psi$ are real-valued. For multi-indices we use the
convention
$$
\alpha = (\alpha _1,\alpha _2)\in \nn {d_1+d_2},
\quad \beta = (\beta _1,\beta _2)\in \nn {d_1+d_2},
\quad \text{and}\quad
\alpha !^s = \alpha _1!^{s_1}\alpha _2!^{s_2}
$$
when $\alpha _j,\beta _j\in \nn {d_j}$ and $s=(s_1,s_2)\in \rr 2$,  $j =1,2$.
Set
$$
R_{\ep ,\alpha ,\beta }(x)
=
x^{\alpha }D_x^\beta 
\left (
\fy (x-y)\psi (y)\, dy - \sum _{k\in \zz d}\fy (x-\ep k)\psi (\ep k)\ep ^d
\right ).
$$
By the mean-value theorem we have for some $\rho _k=\rho _k(x,y)\in Q_{d,1}$,
$k\in \zz d$ that
\begin{multline*}
|R_{\ep ,\alpha ,\beta }(x)|
=
\left |
\int _{\rr d}
x^{\alpha }(D_x^\beta \fy )(x-y)\psi (y)\, dy
-
\sum _{k\in \zz d}x^{\alpha }(D_x^\beta \fy )(x-\ep k)\psi (\ep k)\ep ^d
\right |
\\[1ex]
=
\left |
\sum _{k\in \zz d}
\left (\int _{\ep k+Q_{d,\ep}}
x^{\alpha }(D_x^\beta \fy )(x-y)\psi (y)\, dy
-
x^{\alpha }(D_x^\beta \fy )(x-\ep k)\psi (\ep k)\ep ^d
\right )
\right |
\\[1ex]
=
\left |
\sum _{k\in \zz d}
\left (
x^{\alpha }(D_x^\beta \fy )(x-\ep k-\ep \rho _k)\psi (\ep k+\ep \rho _k)
-
x^{\alpha }(D_x^\beta \fy )(x-\ep k)\psi (\ep k)
\right ) \ep ^d
\right |
\\[1ex]
\le
\sum _{k\in \zz d}
\left |
x^{\alpha }(D_x^\beta \fy )(x-\ep k-\ep \rho _k)\psi (\ep k+\ep \rho _k)
-
x^{\alpha }(D_x^\beta \fy )(x-\ep k)\psi (\ep k)
\right | \ep ^d
\\[1ex]
\le
\sum _{k\in \zz d}\sum _{j=1}^d \sup _{z\in Q_{d,\ep}}
\left |
D_{z_j}\left ( x^{\alpha }(D_x^\beta \fy )(x-\ep k-z)\psi (\ep k+z)
\right ) \right | \ep ^{d+1}
\le J_1+J_2,
\end{multline*}
where
\begin{align*}
J_1 &= \sum _{\gamma \le \alpha} \sum _{j=1}^d\sum _{k\in \zz d}
{\alpha \choose \gamma}
\sup _{y\in \ep k+Q_{d,\ep}}
\left |
(x-y)^{\gamma}D_{x_j}D_x^\beta \fy (x-y)y^{\alpha -\gamma}
\psi (y) \right | \ep ^{d+1}
\intertext{and}
J_2 &= \sum _{\gamma \le \alpha} \sum _{j=1}^d\sum _{k\in \zz d}
{\alpha \choose \gamma}
\sup _{y\in \ep k+Q_{d,\ep}}
\left |
(x-y)^{\gamma}D_x^\beta \fy (x-y)y^{\alpha -\gamma}
D_{y_j}\psi (y) \right | \ep ^{d+1}.
\end{align*}

\par

Since $(m+1)!\le 2^mm!$ when $m\ge 0$ is an integer, and
$\fy ,\psi \in \maclS _{s_1,s_2}^{\sigma _1,\sigma _2}
(\rr {d_1+d_2})$, we get
\begin{multline*}
J_j
\lesssim
h^{|\alpha +\beta |}2^{\sigma |\beta |}d
\sum _{\gamma \le \alpha} \sum _{k\in \zz d}
{\alpha \choose \gamma}
\gamma !^s\beta !^\sigma 
\sup _{y\in \ep k+Q_{d,\ep}}
\left |
y^{\alpha -\gamma}
e^{-2r(|y_1|^{\frac 1{s_1}} + |y_2|^{\frac 1{s_2}} ) }
\right |
\ep ^{d+1}
\\[1ex]
\lesssim
\ep h^{|\alpha +\beta |}2^{\sigma |\beta |}
\sum _{\gamma \le \alpha} \sum _{k\in \zz d}
{\alpha \choose \gamma}
\gamma !^s\beta !^\sigma 
(\alpha -\gamma )!^s
e^{-r(|\ep k_1|^{\frac 1{s_1}} + |\ep k_2|^{\frac 1{s_2}} ) }
\ep ^d
\\[1ex]
\lesssim
\ep (2^\sigma 2h)^{|\alpha +\beta |} \alpha !^s
\beta !^\sigma 
\sum _{k\in \zz d}
e^{-r(|\ep k_1|^{\frac 1{s_1}} + |\ep k_2|^{\frac 1{s_2}} ) }
\ep ^d
\lesssim 
\ep (2^\sigma 2h)^{|\alpha +\beta |} \alpha !^s
\beta !^\sigma  ,
\end{multline*}
$j=1,2$, for some positive constants $h$ and $r$. This implies
that for some $h>0$ we have
\begin{equation}\label{Eq:(A.1)} 
\sup _{\alpha ,\beta \in \nn d}\left (
\frac {\nm {R_{\ep ,\alpha ,\beta }}{L^\infty }}
{h^{|\alpha +\beta |}\alpha !^s\beta !^\sigma}
\right )
\le C\ep
\end{equation}
for some positive constants $C$ and $h$ which are independent of $\ep$.

\par

Since the right-hand side tends to zero when $\ep >0$ tends to zero,
the stated convergence follows in this case.

\par

The general case follows from the previous case, after writing
$\fy = \fy _1+i\fy _2$
and $\psi = \psi _1+i\psi _2$ with $\fy _{j}$ and $\psi _{j}$ being 
real-valued, $j=1,2$,
giving that $\fy *\psi$ is a superposition of $\fy _{j_1}*\psi _{j_2}$,
$j_1,j_2\in \{ 1,2\}$, and using the fact that $\fy _j\in
\maclS _{s_1,s_2}^{\sigma _1,\sigma _2}(\rr {d_1+d_2})$ when
$\fy \in \maclS _{s_1,s_2}^{\sigma _1,\sigma _2}(\rr {d_1+d_2})$.
\end{proof}

\par

We may now prove the following result related to \cite[Theorem 4.1.2]{Ho1}

\par

\begin{lemma}\label{Lemma:A.2}
Let $s_1,s_2,\sigma _1,\sigma _2>0$, $\fy ,\psi \in \maclS
_{s_1,s_2}^{\sigma _1,\sigma _2}(\rr {d_1+d_2})$ and let $f\in (\maclS
_{s_1,s_2}^{\sigma _1,\sigma _2})'(\rr {d_1+d_2})$. Then
\begin{equation}\label{Eq:ConvAsociative}
(f*\fy )*\psi = f*(\fy *\psi ).
\end{equation}

\par

The same holds true if each
$\maclS _{t_1,t_2}^{s_1,s_2}$ and $(\maclS _{t_1,t_2}^{s_1,s_2})'$
are replaced by
$\Sigma _{t_1,t_2}^{s_1,s_2}$ and $(\Sigma _{t_1,t_2}^{s_1,s_2})'$,
respectively.
\end{lemma}

\par

\begin{proof}
We use the same notations in the previous proof.
Since the Riemann sum in Lemma \ref{Lemma:A.1}
converges to $\fy *\psi$ in $\maclS _{s_1,s_2}
^{\sigma _1,\sigma _2}$, we get
\begin{multline*}
(f*(\fy *\psi ))(x) = \lim _{\ep \to 0}
\left \langle 
f, \sum _{k\in \zz d}\fy (x-\cdo -\ep k)\psi (\ep k)
\ep ^d \right \rangle 
\\[1ex]
=
\lim _{\ep \to 0}
\left (
\sum _{k\in \zz d}(f*\fy )(x-\ep k)\psi (\ep k)\ep ^d
\right )
.
\end{multline*}
Here the second equality follows by the fact that
$$
y\mapsto \sum _{k\in \zz d}\fy (x-y -\ep k)\psi (\ep k)
$$
converges in $\maclS _{t_1,t_2}^{s_1,s_2}$.

\par

We have that $f*\fy$ is smooth, and for some
$r_0>0$ we have
$$
|(f*\fy )(x-\ep k)\psi (\ep k)| \lesssim
e^{r(|x_1-\ep k_1|^{\frac 1{s_1}} + |x_2-\ep k_2|^{\frac 1{s_2}})}
e^{-2r_0(|\ep k_1|^{\frac 1{s_1}} + |\ep k_2|^{\frac 1{s_2}})}
$$
for every $r>0$. This gives
$$
|(f*\fy )(x-\ep k)\psi (\ep k)| \le C_xe^{-r_0(|\ep k_1|^{\frac 1{s_1}}
+ |\ep k_2|^{\frac 1{s_2}})},
$$
for some constant $C_x$ which only depends on $x$ and $r_0$.
It follows that
$$
\sum _{k\in \zz d}(f*\fy )(x-\ep k)\psi (\ep k)\ep ^d
$$
is a Riemann sum which converges to
$$
\int (f*\fy )(x-y)\psi (y)\, dy=((f*\fy )*\psi ) (x).
$$
Hence \eqref{Eq:ConvAsociative} holds, and the result follows.
\end{proof}

\par

By the previous lemma it is now straight-forward to prove the
following.

\par

\begin{lemma}\label{Lemma:A.3}
Let $\mabfs ,\mabfsi \in \rr n_+$, $d=d_1+\cdots d_n$ and suppose $f\in (\maclS
_{\mabfs}^{\mabfsi})'(\rr {d})$ satisfies
$\scal f{\fy _1\otimes \cdots \otimes \fy _n}=0$ for every
$\fy _j\in \maclS _{s_j}^{\sigma _j}(\rr {d_j})$, $j=1,\dots ,n$. Then
$f=0$.

\par

The same holds true if each $\maclS _{s_j}^{\sigma _j}$, $(\maclS _{s_j}^{\sigma _j})'$,
$\maclS _{\mabfs}^{\mabfsi}$ and $(\maclS _{\mabfs}^{\mabfsi})'$
are replaced by $\Sigma _{s_j}^{\sigma _j}$, $(\Sigma _{s_j}^{\sigma _j})'$,
$\Sigma _{\mabfs}^{\mabfsi}$ and $(\Sigma _{\mabfs}^{\mabfsi})'$,
respectively.
\end{lemma}

\par

\begin{proof}
We only prove the result in the Roumieu case. The Beurling case follows by
similar arguments and is left for the reader.
We use the same notations as in the previous proofs.

\par

First suppose $n=2$. Let
$\fy \in \maclS _{s_1,s_2}^{\sigma _1,\sigma _2}(\rr {d_1+d_2})$, $\phi _j\in
\maclS _{s_j}^{\sigma _j}(\rr {d_j})$ be such that
$$
\int _{\rr {d_j}}\phi _j(x_j)\, dx_j =1,
$$
and let
$$
\phi _\ep =|\ep | ^{-(d_1+d_2)}(\phi _1\otimes \phi _2)(\ep ^{-1}\cdo ),
$$
when $\ep$ is real. Then the assumptions implies that $\check f *\phi _\ep =0$
for every $\ep$. Here $\check f$ is defined by $\check f(x)= f(-x)$.
By Lemma \ref{Lemma:A.2} we get
$$
\scal f\fy = \lim _{\ep \to 0} \scal f{\phi _\ep *\fy}
= \lim _{\ep \to 0}(\check f*(\phi _\ep *\fy))(0) = \lim _{\ep \to 0}
((\check f*\phi _\ep ) *\fy)(0),
$$
and the result follows for $n=2$.

\par

For general $n\ge 2$, the result follows from the case $n=2$ and induction.
The details are left for the reader.
\end{proof}

\par

\begin{proof}[Proof of Theorem \ref{Prop:GSTensor}]
We only prove the result in the Roumieu cases.
The Beurling cases
follow by similar arguments and are left for the reader.

\par

By straight-forward computations it follows that
$$
\fy \mapsto \scal {f_1}{\psi _1}
\quad \text{and}\quad
\fy \mapsto \scal {f_2}{\psi _2}
$$
define continuous linear forms $g_1$ and $g_2$ on 
$\maclS _{s_1,s_2}^{\sigma _1,\sigma _2}(\rr {d_1+d_2})$. Hence $g_1,g_2
\in (\maclS _{s_1,s_2}^{\sigma _1,\sigma _2})'(\rr {d_1+d_2})$.
It is clear that both $g_1$ and $g_2$ in place of $f$ satisfy
\eqref{Eq:TensorTensor1},
and the existence of $f$ follows.

\par

If $f\in (\maclS _{s_1,s_2}^{\sigma _1,\sigma _2})'(\rr {d_1+d_2})$ is
arbitrary such that \eqref{Eq:TensorTensor1} holds, then
$$
\scal {f-g_j}{\fy _1\otimes \fy _2} = \scal {f_1}{\fy _1}\scal {f_2}{\fy _2}
-\scal {f_1}{\fy _1}\scal {f_2}{\fy _2} = 0,
$$
and Lemma \ref{Lemma:A.3}  shows that $f=g_1=g_2$. This gives
the uniqueness of $f$, as well as \eqref{Eq:TensorTensor2}.
\end{proof}

\par

In order to consider corresponding multi-linear situation of
Theorem \ref{Prop:GSTensor}, we let $S_n$ be the permutation group
of $\{ 1,\dots ,n\}$, and let inductively
\begin{align}
\fy _{n,\tau} (x_{\tau (1)},\dots ,x_{\tau (n)})
&=
\fy (x_1,\dots ,x_n),\qquad x_j\in \rr {d_j},\ \tau \in S_n,
\label{Eq:DefFyn}
\intertext{and}
\fy _{j,\tau} (x_{\tau (1)},\dots ,x_{\tau (j)})
&=
\scal {f_{\tau (j+1)}}
{\fy _{j+1,\tau}(x_{\tau (1)},\dots ,x_{\tau (j)},\cdo)}
\label{Eq:DefFyj}
\end{align}
when $f_j$ for $j=1,\dots ,n$ are suitable distributions
and $\fy$ is a suitable function. Then
Theorem \ref{Prop:GSTensor} can be reformulated as follows. It is
also convenient to set
\begin{equation}\label{Eq:stauDef}
\begin{gathered}
\mabfs _{j,\tau} = (s_{\tau (1)},\dots ,s_{\tau (j)}),\qquad
\mabfsi _{j,\tau} = (\sigma _{\tau (1)},\dots ,\sigma _{\tau (j)})
\\[1ex]
\text{and}\quad
d_{j,\tau} = d_{\tau (1)}+\cdots + d_{\tau (j)},
\end{gathered}
\end{equation}
when $j=1,\dots ,n$ and $\mabfs ,\mabfsi \in \rr n_+$,

\par

\begin{thm}\label{Prop:GSTensorM}
Let $\tau \in S_2$, $d=d_1+d_2$, $\mabfs ,\mabfsi \in \rr 2_+$,
$d_{j,\tau}$, $\mabfs _{j,\tau}$ and $\mabfsi _{j,\tau}$ be as in
\eqref{Eq:stauDef}, $f_j\in (\maclS _{s_j}^{\sigma _j})'(\rr {d_j})$,
$\fy \in \maclS _{\mabfs}^{\mabfsi}(\rr {d})$
and let $\fy _{j,\tau}$ be given by \eqref{Eq:DefFyn} and
\eqref{Eq:DefFyj}, $j=1,2$.
Then $\fy _{j,\tau}\in  \maclS _{\mabfs _{j\tau}} ^{\mabfsi _{j,\tau}}
(\rr {d_{j,\tau}})$, and
there is a unique distribution $f$ in
$(\maclS _{\mabfs}^{\mabfsi})'(\rr {d})$ 
such that
for every $\fy _j\in \maclS _{s_j}^{\sigma _j}(\rr {d_j})$, $j=1,\dots ,n$, and
$\fy _2\in \maclS _{s_2}^{\sigma _2}(\rr {d_2})$,
\begin{equation}\label{Eq:TensorSplitFubbini}
\scal f{\fy _1\otimes \fy _2} =  \prod _{k=1}^2 \scal {f_k}{\fy _k}
\quad \text{and}\quad
\scal f{\fy}=\scal {f_{\tau (1)}}{\fy _{1,\tau}}
\end{equation}
hold.

\medspace

The same holds true with 
$\Sigma _{s_j}^{\sigma _j}$, $(\Sigma _{s_j}^{\sigma _j})'$,
$\Sigma _{\mabfs}^{\mabfsi}$ and
$(\Sigma _{\mabfs}^{\mabfsi})'$
in place of
$\maclS _{s_j}^{\sigma _j}$, $(\maclS _{s_j}^{\sigma _j})'$,
$\maclS _{\mabfs}^{\mabfsi}$ and
$(\maclS _{\mabfs}^{\mabfsi})'$,
respectively, at each occurrence.
\end{thm}

\par

Here the second equality in \eqref{Eq:TensorSplitFubbini} is the same as
the Fubbini property \eqref{Eq:TensorTensor2}. The multi-linear version of
the previous theorem is the following, and follows by similar arguments as
for its proof. The details are left for the reader.

\par

\renewcommand{\rubrik}{Theorem \ref{Prop:GSTensorM}$'$}

\par

\begin{tom}
Let $\tau \in S_n$, $d=d_1+\dots +d_n$, $\mabfs ,\mabfsi \in \rr n_+$,
$d_{j,\tau}$, $\mabfs _{j,\tau}$ and $\mabfsi _{j,\tau}$ be as in
\eqref{Eq:stauDef}, $f_j\in (\maclS _{s_j}^{\sigma _j})'(\rr {d_j})$,
$\fy \in \maclS _{\mabfs}^{\mabfsi}(\rr {d})$
and let $\fy _{j,\tau}$ be given by \eqref{Eq:DefFyn} and
\eqref{Eq:DefFyj}, $j=1,\dots ,n$.
Then $\fy _{j,\tau}\in  \maclS _{\mabfs _{j\tau}} ^{\mabfsi _{j,\tau}}
(\rr {d_{j,\tau}})$, and
there is a unique distribution $f$ in
$(\maclS _{\mabfs}^{\mabfsi})'(\rr {d})$ 
such that
for every $\fy _j\in \maclS _{s_j}^{\sigma _j}(\rr {d_j})$, $j=1,\dots ,n$,
\begin{equation}\tag*{(\ref{Eq:TensorSplitFubbini})$'$}
\scal f{\fy _1\otimes \cdots \otimes \fy _n} =  \prod _{k=1}^n
\scal {f_k}{\fy _k}
\quad \text{and}\quad
\scal f{\fy}=\scal {f_{\tau (1)}}{\fy _{1,\tau}}
\end{equation}
hold.

\medspace

The same holds true with
$\Sigma _{s_j}^{\sigma _j}$, $(\Sigma _{s_j}^{\sigma _j})'$,
$\Sigma _{\mabfs}^{\mabfsi}$ and
$(\Sigma _{\mabfs}^{\mabfsi})'$
in place of
$\maclS _{s_j}^{\sigma_j}$, $(\maclS _{s_j}^{\sigma _j})'$,
$\maclS _{\mabfs}^{\mabfsi}$ and
$(\maclS _{\mabfs}^{\mabfsi})'$,
respectively, at each occurrence.
\end{tom}

\par

\begin{example}\label{Ex:STFTforGS}
Let $s,\sigma >0$. An important object in time-frequency and micro-local analysis concerns
the short-time Fourier transform. If $\phi \in \maclS _s^\sigma (\rr d)\setminus \{ 0\}$
is fixed, then the short-time Fourier transform of $f\in \maclS _s^\sigma (\rr d)$ is
defined by
$$
V_\phi f(x,\xi ) = (2\pi )^{-\frac d2}\int _{\rr d}f(y)\overline{\phi (y-x)}e^{-i\scal y\xi}\, dy.
$$
It follows that
\begin{align}
V_\phi f(x,\xi ) &= (2\pi )^{-\frac d2}\scal f{\overline{\phi (\cdo -x)}e^{-i\scal \cdo \xi}}
\label{Eq:STFTident1}
\intertext{and}
V_\phi f(x,\xi ) &= \mascF (f\cdot \overline{\phi (\cdo -x)})(\xi )
\label{Eq:STFTident2}
\end{align}
for such choices of $\phi$ and $f$.

\par

We notice that the right-hand side of \eqref{Eq:STFTident1} also makes
sense as a smooth function on $\rr {2d}$ if
the assumption on $f$ is relaxed into $f\in (\maclS _s^\sigma )'(\rr d)$. For such
$f$ we therefore let \eqref{Eq:STFTident1} define the short-time Fourier transform of
$f$ with respect to $\phi$.
Since the map which takes $\phi$ into $y\mapsto \phi (y-x)e^{i\scal y\xi}$ is
continuous and smooth with respect to $(x,\xi)$ from $\maclS _s^\sigma (\rr d)$ to itself 
it follows that $V_\phi f$ is smooth.
By \cite[Proposition 2.2]{Toft15} it follows that
$V_\phi f$ belongs to $(\maclS _{s,\sigma}^{\sigma ,s})'(\rr {2d})$. Consequently,
$$
V_\phi f\in (\maclS _{s,\sigma}^{\sigma ,s})'(\rr {2d})\cap C^\infty (\rr {2d}).
$$

\par

Let $U$ be the operator which takes any $F(x,y)$ into $F(y,y-x)$ and recall that
$\mascF _2F$ is the partial Fourier transform of $F(x,y)$ with respect to the $y$ variable.
Then the right-hand side of \eqref{Eq:STFTident2} equals
\begin{equation}\label{Eq:ReformSTFT}
(\mascF _2(U(f\otimes \overline \phi )))(x,\xi ).
\end{equation}
We notice that the right-hand side makes sense as an element in 
$(\maclS _{s,\sigma}^{\sigma ,s})'(\rr {2d})$ for any $f,\phi \in (\maclS _s^\sigma )'(\rr d)$
in view of Remark \ref{Rem:FTExtentions},
which may be used to extend the definition of the short-time Fourier transform to
even more general situations.

\par

We claim that the right-hand sides of \eqref{Eq:STFTident1} and \eqref{Eq:STFTident2}
agree when $f \in (\maclS _s^\sigma )'(\rr d)$ and $\phi \in \maclS _s^\sigma (\rr d)$.

\par

In fact, let $\psi \in \maclS _{s,\sigma}^{\sigma ,s}(\rr {2d})$ and set
\begin{align*}
\fy (x,\xi ,y) &\equiv \psi (x,\xi )\overline{\phi (y-x)}e^{-i\scal y\xi}\in \maclS _{s,\sigma ,s}
^{\sigma ,s,\sigma}(\rr {3d}),
\\[1ex]
F &\equiv 1_{\rr {2d}}\otimes f = 1_{\rr {d}}\otimes 1_{\rr {d}}\otimes f \in (\maclS _{s,\sigma ,s}
^{\sigma ,s,\sigma})'(\rr {3d}),
\\[1ex]
\fy _1(x,y) &\equiv
\int _{\rr {d}} \fy (x,\xi ,y)\, d\xi
=
\int _{\rr {d}} \psi (x,\xi )\overline {\phi (y-x)}e^{-i\scal y \xi}\, dxd\xi ,
\\[1ex]
\fy _2(x,\xi ) &\equiv
\scal f{\psi (x,\xi )\overline {\phi (\cdo -x)}e^{-i\scal \cdo \xi}} = \psi (x,\xi )\cdot V_\phi f(x,\xi ) ,
\end{align*}
and let $g$ be the right-hand side of \eqref{Eq:STFTident2}.
By the Fubbini property at the right-hand of \eqref{Eq:TensorSplitFubbini}$'$
we get
\begin{equation}\label{Eq:TillampFubb1}
\scal F\fy = \scal {1_{\rr d}\otimes f}{\fy _1} = \scal g\psi
\end{equation}
and
\begin{equation}\label{Eq:TillampFubb2}
\scal F\fy = \scal {1_{\rr {2d}}}{\fy _2} = \scal {V_\phi f}\psi .
\end{equation}
Since $\psi$ was arbitrarily chosen, it follows that $g=V_\phi f$ in
$(\maclS _{s,\sigma}^{\sigma ,s})'(\rr {2d})$, and the claim follows.
\end{example}

\par

\section{Tensor product of Pilipovi{\'c}
spaces}\label{sec3}

\par

In this section we discuss the tensor map on Pilipovi{\'c}
spaces. Especially we prove Theorem \ref{Prop:GSTensor}.
Thereafter we deduce a multi-linear version of this result.

\par

First we show that the tensor map possess natural
mapping properties on Pilipovi{\'c} spaces.

\par

\begin{prop}
Let $s\in \overline {\mathbf R_\flat}$. Then the following is true:
\begin{enumerate}
\item the map $(f_1,f_2)\mapsto f_1\otimes f_2$ from $\mascS (\rr {d_1})
\times \mascS (\rr {d_2})$ to $\mascS (\rr {d_1+d_2})$, restricts to a continuous
map from $\maclH _s(\rr {d_1})
\times \maclH _s (\rr {d_2})$ to $\maclH _s (\rr {d_1+d_2})$;

\vrum

\item the map $(f_1,f_2)\mapsto f_1\otimes f_2$ from $\mascS (\rr {d_1})
\times \mascS (\rr {d_2})$ to $\mascS (\rr {d_1+d_2})$, restricts to a continuous
map from $\maclH _{0,s}(\rr {d_1})
\times \maclH _{0,s} (\rr {d_2})$ to $\maclH _{0,s} (\rr {d_1+d_2})$.
\end{enumerate}
\end{prop}

\par

\begin{proof}
We only prove (1) and in the case $s>0$.
The case $s=0$ and  (2) follow by similar arguments and are
left for the reader. If
$$
f_j= \sum _{\alpha _j\in \nn {d_j}} c_{\alpha _j}(f_j)h_{\alpha _j},
$$
then
$$
f=\sum _{\alpha \in \nn d}c_\alpha h_\alpha ,\qquad c_\alpha 
= c_{\alpha _1}(f_1)c_{\alpha _2}(f_2),\qquad \alpha =(\alpha _1,\alpha _2),\
\alpha _j\in \nn {d_j},\ j=1,2.
$$
If $s\in \mathbf R_+$, then
$$
|c_{\alpha _j}(f_j)| \lesssim e^{-c |\alpha _j|^{\frac 1{2s}}},
$$
for some $c>0$. This gives
$$
|c_\alpha | \lesssim e^{-c (|\alpha _1|^{\frac 1{2s}}+|\alpha _2|^{\frac 1{2s}}}
\le e^{-c |\alpha |^{\frac 1{2s}}/(1+2^{\frac 1s})},\quad \alpha =(\alpha _1,\alpha _2),
$$
and it follows that $f\in \maclH _s(\rr d)$.

\par

If instead $s=\flat _\sigma$, for some $\sigma >0$, then
$$
|c_{\alpha _j} (f_j)|\lesssim r^{|\alpha _j|}\alpha _j!^{\frac 1{2\sigma }},\qquad j=1,2,
$$
for some $r>0$. Hence, if $\alpha =(\alpha _1,\alpha _2)$, we get
$$
|c_\alpha |\lesssim r^{|\alpha |}(\alpha _1!\alpha _2!)^{\frac 1{2\sigma }}
\le
((2^{\frac 1s}+1)r)^{|\alpha |}\alpha !^{\frac 1{2\sigma}},
$$
for every $r>0$, and it follows that $f\in \maclH _s(\rr d)$ in this case as well. From these
estimates it also follows that the map $(f_1,f_2)\to f_1\otimes f_2$ is continuous from
$\maclH _s(\rr {d_1}) \times \maclH _s (\rr {d_2})$ to $\maclH _s (\rr {d_1+d_2})$, and
the result follows.
\end{proof}

%

\par

\begin{proof}[Proof of Theorem \ref{Thm:TensorprodDistr}]
Let $d=d_1+d_2$. We shall deal with the Hermite sequence representations
of the elements in
the Pilipovi{\'c} spaces. Such approach is performed in \cite{ReSi}, when deducing tensor
product and kernel results for tempered distributions.
We only prove the results when $f_j\in \maclH _s'(\rr d)$ and $s>0$.
The cases when $f_j\in \maclH _{0,s}'(\rr d)$ or $s=0$ follow by similar arguments and are
left for the reader.

\par

First we prove the uniqueness. Suppose that both $f,g\in \maclH _s'(\rr {d})$ satisfy
\begin{equation}\label{Eq:ActsOnTensorprods2}
\scal f{\phi _1\otimes \phi _2} = \scal g{\phi _1\otimes \phi _2}
= \scal {f_1}{\phi _1} \scal {f_2}{\phi _2},
\end{equation}
and let $c_\alpha (f)$ and $c_\alpha (g)$ be their Hermite coefficients of order $\alpha \nn d$.
By choosing $\phi _1=h_{\alpha _1}$ and $\phi _2=h_{\alpha _2}$, \eqref{Eq:ActsOnTensorprods2}
implies $c_\alpha (f)=c_\alpha (g)$ when $\alpha =(\alpha _1,\alpha _2)$. Consequently,
$f=g$, and the uniqueness follows.

\par

We have
$$
f_j= \sum _{\alpha _j\in \nn {d_j}} c_{\alpha _j}(f_j)h_{\alpha _j},
$$
where $c_{\alpha _j}(f_j)$ for every $\alpha _j\in \nn {d_j}$ are unique and equal to
$(f_j,h_{\alpha _j})$, $j=1,2$. 

\par

Now let $f$ be the element in $\maclH _0'(\rr {d})$, $d=d_1+d_2$ with expansion
$$
f=\sum _{\alpha \in \nn d}c_\alpha h_\alpha ,
$$
where
$$
c_\alpha = c_{\alpha _1}(f_1)c_{\alpha _2}(f_2),\qquad \alpha =(\alpha _1,\alpha _2),\
\alpha _j\in \nn {d_j},\ j=1,2.
$$
We claim that $f\in \maclH _s'(\rr d)$.

\par

In fact, if $s\in \mathbf R_+$, then
\begin{equation}\label{Eq:HermitejCoeffEst}
|c_{\alpha _j}(f_j)| \lesssim e^{\ep |\alpha _j|^{\frac 1{2s}}}
\end{equation}
for every $\ep >0$, and it follows that
\begin{equation*}
|c_\alpha | \lesssim e^{\ep (|\alpha _1|^{\frac 1{2s}}+|\alpha _2|^{\frac 1{2s}} )}
\le e^{2\ep |\alpha |^{\frac 1{2s}}},\quad \alpha =(\alpha _1,\alpha _2),
\end{equation*}
for every $\ep >0$. This is the same as $f\in \maclH _s'(\rr d)$.

\par

If instead $s=\flat _\sigma$, for some $\sigma >0$, then
$$
|c_{\alpha _j} (f_j)|\lesssim r^{|\alpha _j|}\alpha _j!^{\frac 1{2\sigma }},\qquad j=1,2,
$$
for every $r>0$. Hence, if $\alpha =(\alpha _1,\alpha _2)$, we get
$$
|c_\alpha |\lesssim r^{|\alpha |}(\alpha _1!\alpha _2!)^{\frac 1{2\sigma }}
=
r^{|\alpha |}\alpha !^{\frac 1{2\sigma}},
$$
for every $r>0$, and it follows that $f\in \maclH _s'(\rr d)$ in this case as well.

\par

If $\fy _j\in \maclH _0(\rr {d_j})$ and $\fy \in \maclH _0(\rr d)$, $j=1,2$, then
\eqref{Eq:TensorTensor1} and \eqref{Eq:TensorTensor2} follow by straight-forward
computations, using the fact that the set of Hermite functions is an orthonormal basis of
$L^2$. For general $\fy _j\in \maclH _s(\rr {d_j})$ and
$\fy \in \maclH _s(\rr d)$, $j=1,2$, the result now follows
from dominating convergence, using the fact that $\maclH _0(\rr d)$ is dense in
$\maclH_s(\rr d)$.
\end{proof}

\par

In order to formulate a multi-linear version of Theorem \ref{Thm:TensorprodDistr}
we first reformulate the result as follows.

\begin{thm}\label{Prop:PilTensorM}
Let $\tau \in S_2$, $d=d_1+d_2$, $s\in \overline{\mathbf R_\flat}$,
$d_{j,\tau}$ be as in
\eqref{Eq:stauDef}, $f_j\in (\maclH _s)'(\rr {d_j})$,
$\fy \in \maclH _s(\rr {d})$
and let $\fy _{j,\tau}$ be given by \eqref{Eq:DefFyn} and
\eqref{Eq:DefFyj}, $j=1,2$.
Then $\fy _{j,\tau}\in  \maclH _s
(\rr {d_{j,\tau}})$, and
there is a unique distribution $f$ in $(\maclH _s)'(\rr {d})$ 
such that
for every $\fy _{j}\in \maclH _s(\rr {d_{j}})$,
$j=1,2$,
$$
\scal f{\fy _1\otimes \fy _2} =  \prod _{k=1}^2 \scal {f_k}{\fy _k}
\quad \text{and}\quad
\scal f{\fy}=\scal {f_{\tau (1)}}{\fy _{1,\tau}}
$$
hold.

\medspace

The same holds true with 
$\maclH _{0,s}$ and
$\maclH _{0,s}'$ in place of
$\maclH _{s}$ and
$\maclH _{s}'$, respectively, at each occurrence.
\end{thm}

\par

The multi-linear version of the previous theorem is the following,
and follows by similar arguments. The details are left for the reader.

\par

\renewcommand{\rubrik}{Theorem \ref{Prop:PilTensorM}$'$}

\par

\begin{tom}
Let $\tau \in S_2$, $d=d_1+\cdots +d_n$,
$s\in \overline{\mathbf R_\flat}$,
$d_{j,\tau}$ be as in
\eqref{Eq:stauDef}, $f_j\in (\maclH _s)'(\rr {d_j})$,
$\fy \in \maclH _s(\rr {d})$
and let $\fy _{j,\tau}$ be given by \eqref{Eq:DefFyn} and
\eqref{Eq:DefFyj}, $j=1,\dots ,n$.
Then $\fy _{j,\tau}\in  \maclH _s
(\rr {d_{j,\tau}})$, and
there is a unique distribution $f$ in $(\maclH _s)'(\rr {d})$ 
such that
for every $\fy _{j}\in \maclH _s(\rr {d_{j}})$,
$j=1,\dots ,n$,
$$
\scal f{\fy _1\otimes \cdots \otimes \fy _n}
=
\prod _{k=1}^n \scal {f_k}{\fy _k}
\quad \text{and}\quad
\scal f{\fy}=\scal {f_{\tau (1)}}{\fy _{1,\tau}}
$$
hold.

\medspace

The same holds true with 
$\maclH _{0,s}$ and
$\maclH _{0,s}'$ in place of
$\maclH _{s}$ and
$\maclH _{s}'$, respectively, at each occurrence.
\end{tom}

\par

\begin{rem}\label{Rem:STFTforPilSpaces}
Only certain parts of the properties in Example \ref{Ex:STFTforGS} carry over
to Pilipovi{\'c} spaces of functions and distributions, in the case when
these spaces do not agree with Gelfand-Shilov spaces of functions
and distributions. (See Remark \ref{Remark:GSHermite}.) In order to deal with
such questions, it it convenient to consider the image of such spaces under
the Bargmann transform, which is defined by
$$
(\mathfrak V_df)(z) = \pi ^{-\frac d4}\scal f{\exp (-\frac 12
\big (\scal zz+|\cdo |^2) +\sqrt 2\scal z\cdo\big )},
$$
when $f$ is a suitable (ultra-)distribution (cf. \cite{Ba,Toft15}).

\par

In fact, let $\maclA _s(\cc d)$ ($\maclA _{0,s}(\cc d)$) be the set of all
$F$ in $A(\cc d)$, the set of entire functions on $\cc d$, which satisfies
$$
|F(z)|\lesssim e^{r(\log \eabs z)^{\frac 1{1-2s}}}
$$ 
when $s<\frac 12$ and
$$
|F(z)| \lesssim e^{r|z|^{\frac {2\sigma}{\sigma +1}}}
$$
when $s=\flat _\sigma$, for some $r>0$ (for every $r>0$). Also let
$\maclA _{0,1/2}(\cc d)$ be the set of all $F\in A(\cc d)$ such that
$|F(z)|\lesssim e^{r|z|^2}$ for all $r>0$.
Then it is proved in \cite{FeGaTo,Toft15} that $\mathfrak V_d$ is bijective from
$\maclH _s(\rr d)$ to $\maclA _s(\cc d)$ when $s\in \mathbf R_\flat$ and $s<\frac 12$,
and from $\maclH _{0,s}(\rr d)$ to $\maclA _{0,s}(\cc d)$ when
$s\in \mathbf R_\flat$ and $s\le \frac 12$.

\par

By straight-forward computations we have
\begin{align*}
(\mathfrak V_d(f(\cdo -x_0)))(z) &= e^{\sqrt 2 \scal z{x_0} +\frac 12|x_0|^2}
(\mathfrak V_df)(z+\sqrt 2 \, x_0)
\intertext{and}
(\mathfrak V_d(fe^{-i\scal \cdo {\xi _0}}))(z) &= e^{-\sqrt 2 \, i\scal z{\xi _0}+\frac 12|\xi _0|^2}
(\mathfrak V_df)(z+i\sqrt 2 \, \xi _0).
\end{align*}
Consequently, by Remark \ref{Remark:GSHermite} and
the mapping properties of the Pilipovi{\'c} spaces above under the Bargmann
transform, it follows that the following is true:
\begin{enumerate}
\item if  $\maclH _s(\rr d)$ and $\maclH _s'(\rr d)$ are invariant under translations and modulations,
if and only if $s\ge \flat _1$;

\vrum

\item if  $\maclH _{0,s}(\rr d)$ and $\maclH _{0,s}'(\rr d)$ are invariant under
translations and modulations, if and only if $s> \flat _1$.
\end{enumerate}
In particular, the short-time Fourier transform
$$
V_\phi f(x,\xi ) =\scal f{\overline {\phi (\cdo -x)}e^{-i\scal \cdo \xi}} 
$$
makes sense as a smooth function when $s\ge \flat _1$, $f\in \maclH _s'(\rr d)$
and $\phi \in \maclH _s(\rr d)$, and when $s> \flat _1$, $f\in \maclH _{0,s}'(\rr d)$
and $\phi \in \maclH _{0,s}(\rr d)$.

\par

On the other hand, for $s<\frac 12$, it seems to be difficult to guarantee that
\eqref{Eq:STFTident2} is true in general, since the map $U$ in Example
\ref{Ex:STFTforGS} seems not to be well-defined on Pilipovi{\'c} spaces
which fail to be Gelfand-Shilov spaces.
\end{rem}

\par


\begin{thebibliography}{99}
\bibitem{Ba} V. Bargmann \emph{On a Hilbert space of analytic
functions and an associated integral transform}, Comm. Pure
Appl. Math., \textbf{14} (1961), 187--214.

\bibitem{CaRo} M. Cappiello, L. Rodino \emph{$\operatorname{SG}$-pseudodifferential
operators and Gelfand-Shilov spaces}, Rocky Mountain J. Math. \textbf{36} (2006),
1117--1148.

\bibitem{CaTo} M. Cappiello, J. Toft \emph{Pseudo-differential operators
in a Gelfand-Shilov setting}, Math. Nachr. \textbf{290} (2017), 738--755.

\bibitem{ChuChuKim} {J. Chung, S.-Y. Chung, D. Kim},
\emph{Characterizations of the Gelfand-Shilov spaces via
Fourier transforms}, Proc. Amer. Math. Soc.
\textbf{124} (1996), 2101--2108.

\bibitem{FeGaTo} C. Fernandez, A. Galbis, J. Toft
\emph{The Bargmann transform and powers of harmonic oscillator on
Gelfand-Shilov subspaces}, RACSAM \textbf{111} (2017), 1-13.

\bibitem{GS} I. M. Gelfand, G. E. Shilov, \emph{Generalized functions, II-III},
Academic Press, NewYork London, 1968.

\bibitem{Ho1} L. H{\"o}rmander \emph{The Analysis of Linear
Partial Differential Operators}, vol {I--III},
Springer-Verlag, Berlin Heidelberg NewYork Tokyo, 1983, 1985.

\bibitem{LozPer} Z. Lozanov Crvenkovi{\'c}, D. Peri{\v s}i{\'c}
\emph{Hermite expansions of elements of Gelfand Shilov spaces
in quasianalytic and non quasianalytic case},
Novi Sad J. Math. \textbf{37} (2007), 129--147. 

\bibitem{Pil1} S. Pilipovi\'c \emph{Generalization of Zemanian spaces
of generalized functions which
have orthonormal series expansions},
SIAM J. Math. Anal. \textbf{17} (1986), 477?484.

\bibitem{Pil2} S. Pilipovi\'c \emph{Structural theorems 
for periodic ultradistributions},
Proc. Amer. Math. Soc. \textbf{98} (1986), 261--266.

\bibitem{Pil3}
S. Pilipovi\'c \emph{Tempered ultradistributions},
Boll. U.M.I. \textbf{7} (1988), 235--251.

\bibitem{ReSi} {M. Reed, B. Simon} \emph{Methods of modern
mathematical physics, I, II}, Academic Press, London New York, 1979.

\bibitem{Toft10} J. Toft \emph{The Bargmann transform on modulation and
Gelfand-Shilov spaces, with applications to Toeplitz and
pseudo-differential operators}, J. Pseudo-Differ. Oper. Appl. 
\textbf{3} (2012), 145--227.

\bibitem{Toft15} J. Toft
\emph{Images of function and distribution spaces under the
Bargmann transform}, J. Pseudo-Differ. Oper. Appl. \textbf{8} (2017), 83--139.

\bibitem{Tre} F. Treves \emph{Topological vector spaces, distributions and kernels},
Academic Press, New York-London, 1967.
\end{thebibliography}
\end{document}